\documentclass[11pt,reqno]{amsart}
\UseRawInputEncoding

\usepackage{latexsym,bm, amsfonts,amsthm,mathrsfs,CJK}
\usepackage[]{amsmath}

\newtheorem{thm}{Theorem}[section]
 
  \newtheorem{rmk}[thm]{Remark}

\newcommand{\dif}{\mathrm{d}} \DeclareMathAlphabet{\mathsfsl}{OT1}{cmss}{m}{sl} \DeclareMathAlphabet{\mathpzc}{OT1}{pzc}{m}{it}

 \def\d"{^{\prime\prime}} \def\d'{^{\prime}}

\begin{document}
	\title[]{Series representations of positive integral powers of pi}
	\thanks{Supported by Doctoral Scientific Research Starting Foundation of Jingdezhen Ceramic University ( No. 102/01003002031), Academic Achievement Re-cultivation Project of Jingdezhen Ceramic University (No. 215/20506277), Science and Technology Research Project of Jiangxi Provincial Department of Education of China (No. GJJ2201041).}
	\date{} \maketitle
	
	\begin{center}
		Mingzhou XU~\footnote{Email: mingzhouxu@whu.edu.cn}   \\
		School of Information Engineering, Jingdezhen Ceramic University, \\
		Jingdezhen 333403, P. R. China
	\end{center}
	
	{\bf Abstract£º} Using a pointwise version of Fej\'{e}r's theorem about Fourier series, we obtain two formulae related to the series representations of positive integral powers of $\pi$. We also check the correctness of our formulae by the applications of the R software.
	
	
	{\bf MR(2010) Subject Classification}\quad 11A67

	\section{Introduction and main results}
	We know (cf. the second paragraph in page 401 of Chen et al. \cite{Chen2000}, Exercise 14 in page 199 of Rudin \cite{Rudin1976}, Choe \cite{Choe1987}, Ewell \cite{Ewell1992}, Gupta \cite{Gupta2017}) that
	$$
	\sum_{n=1}^{\infty}\frac{(-1)^{n-1}}{(2n-1)}=\frac{\pi}{4}, \sum_{n=1}^{\infty}\frac{(-1)^{n-1}}{(2n-1)^3}=\frac{\pi^3}{32},\sum_{n=1}^{\infty}\frac{(-1)^{n-1}}{(2n-1)^5}=\frac{5\pi^5}{1536},
	$$
	$$
	\sum_{n=1}^{\infty}\frac{1}{n^2}=\frac{\pi^2}{6},\sum_{n=1}^{\infty}\frac{1}{n^4}=\frac{\pi^4}{90},\sum_{n=1}^{\infty}\frac{1}{n^6}=\frac{\pi^6}{945}.
	$$
	For the history of series representations for some powers of $\pi$ and related topics, the interested reader could refer to Borwein \cite{Borwein2000},  Alzer et al. \cite{Alzer2006} and references therein.
	it is interesting to wonder whether or not we could obtain series representations of all positive integral powers of $\pi$ similar to above equalities. Here we present two formulae, which could be applied to produce series representations of all positive integral powers of $\pi$. First, we cite a pointwise version of Fej\'{e}r's theorem about Fourier series as follows ( cf. Theorem 16.2.2 in page 398 of Chen et al. \cite{Chen2000} or Exercise 16 in page 208 of Rudin \cite{Rudin1976}).
	
	\begin{thm}\label{thm01}Suppose a function $f(x)$ with period $2T$ is Riemann-integrable on $[-T, T]$ ( cf. definition in page 121 of Rudin \cite{Rudin1976}), and $f(x+)\doteq\lim_{y\searrow x,y\not=x}f(y)$, $f(x-)\doteq\lim_{y\nearrow x,y\not=x}f(y)$ exist for some $x$. Then
		\begin{equation*}
			\lim_{N\rightarrow\infty}a_0/2+\sum_{n=1}^{N}a_n\cos(n\pi x/T)+b_n\sin(n\pi x/T)=\frac{f(x+)+f(x-)}{2},
		\end{equation*}
		where
		$$
		a_n=\frac1T\int_{-T}^{T}f(x)cos(n\pi x/T)\dif x, n=0,1,2,\ldots,
		$$
		$$
		b_n=\frac1T\int_{-T}^{T}f(x)sin(n\pi x/T)\dif x, n=1,2,\ldots.
		$$
		
	\end{thm}
	
	Our main results are as follows.
	\begin{thm}\label{thm02}
		For any integer $k\ge 1$,
		\begin{align}\label{01}
			\frac12=\frac{1}{4k+2}+\sum_{l=1}^{k}\sum_{n=1}^{\infty}\frac{(2k)!}{(2k+1-2l)!(n\pi)^{2l}}(-1)^{l-1}.
		\end{align}
	\end{thm}
	\begin{thm}\label{thm03}
		For any integer $k\ge 1$,
		\begin{align}\label{02}
			\nonumber-\frac{1}{4k+2}\pi^{2k}&=\sum_{l=0}^{k-1}\sum_{n=1}^{\infty}\pi^{2k-2l-2}\frac{(2k)!(-1)^n(-1)^l}{(2k-2l-1)!(2n)^{2l+2}}\\
			&\quad+\sum_{l=0}^{k}\sum_{n=1}^{\infty}\frac{\pi^{2k-2l-1}(-1)^{n}(-1)^l(2k)!}{(2n-1)^{2l+1}(2k-2l)!}+\sum_{n=1}^{\infty}\frac{(-1)^k(-1)^n(2k)!}{\pi(2n-1)^{2k+1}}.
		\end{align}
	\end{thm}
	\begin{rmk}\label{rmk01}
		By Theorem \ref{thm02}, together with programming in R language, we could obtain that
		$$
		\sum_{n=1}^{\infty}\frac{1}{n^2}=\frac{\pi^2}{6},\sum_{n=1}^{\infty}\frac{1}{n^4}=\frac{\pi^4}{90},
		$$
		$$
		\sum_{n=1}^{\infty}\frac{1}{n^6}=\frac{\pi^6}{945},\sum_{n=1}^{\infty}\frac{1}{n^8}=\frac{\pi^8}{9450},
		$$
		$$
		\sum_{n=1}^{\infty}\frac{1}{n^{10}}=\frac{\pi^{10}}{93555},\sum_{n=1}^{\infty}\frac{1}{n^{12}}=\frac{691\pi^{12}}{638512875},
		$$
		$$
		\sum_{n=1}^{\infty}\frac{1}{n^{14}}=\frac{2\pi^{14}}{18243225},\sum_{n=1}^{\infty}\frac{1}{n^{16}}=\frac{13147\pi^{16}}{1183635518797},
		$$
		$$
		\sum_{n=1}^{\infty}\frac{1}{n^{18}}=\frac{1482\pi^{18}}{1316874094457},\sum_{n=1}^{\infty}\frac{1}{n^{20}}=\frac{19879\pi^{20}}{174337804837681}.
		$$
	\end{rmk}\begin{rmk}
		By (\ref{01}), (\ref{02}) and the proofs of Remarks \ref{rmk01} \ref{rmk02}, we could deduce that any integer $k\ge 1$,
		$$
		\sum_{n=1}^{\infty}\frac{1}{n^{2k}}=\pi^{2k}Q_1,\sum_{n=1}^{\infty}\frac{(-1)^{n-1}}{(2n-1)^{2k+1}}=\pi^{2k+1}Q_2
		$$
		where $Q_1$ and $Q_2$ are both some rational numbers.
	\end{rmk}
	\begin{rmk}
		By R software, we could compute that
		$$
		\sum_{n=1}^{\infty}\frac{1}{n^2}\approx 1.644934,
		\frac{\pi^2}{6}\approx 1.644934,\sum_{n=1}^{\infty}\frac{1}{n^4}\approx1.082323,
		\frac{\pi^4}{90}\approx 1.082323,
		$$
		$$
		\sum_{n=1}^{\infty}\frac{1}{n^6}\approx1.017343,
		\frac{\pi^6}{945}\approx1.017343,\sum_{n=1}^{\infty}\frac{1}{n^8}\approx1.004077,\frac{\pi^8}{9450}\approx1.004077,
		$$
		$$
		\sum_{n=1}^{\infty}\frac{1}{n^{10}}\approx1.000995,\frac{\pi^{10}}{93555}\approx1.000995,\sum_{n=1}^{\infty}\frac{1}{n^{12}}\approx1.000246,\frac{691\pi^{12}}{638512875}\approx1.000246.
		$$

	\end{rmk}
	\begin{rmk}\label{rmk02}
		By Theorems \ref{thm01}-\ref{thm02}, together with $\sum_{n=1}^{\infty}\frac{(-1)^{n-1}}{(2n-1)}=\frac{\pi}{4}$ and programming in R language, we can also obtain
		$$
		\sum_{n=1}^{\infty}\frac{(-1)^{n-1}}{(2n-1)^3}=\frac{\pi^3}{32},\sum_{n=1}^{\infty}\frac{(-1)^{n-1}}{(2n-1)^5}=\frac{5\pi^5}{1536},
		$$
		$$
		\sum_{n=1}^{\infty}\frac{(-1)^{n-1}}{(2n-1)^7}=\frac{61\pi^7}{184320},\sum_{n=1}^{\infty}\frac{(-1)^{n-1}}{(2n-1)^9}=\frac{277\pi^9}{8257536},
		$$
		$$
		\sum_{n=1}^{\infty}\frac{(-1)^{n-1}}{(2n-1)^{11}}=\frac{18269\pi^{11}}{5374843438},\sum_{n=1}^{\infty}\frac{(-1)^{n-1}}{(2n-1)^{13}}=\frac{5071747\pi^{13}}{14726725711261},
		$$
		$$
		\sum_{n=1}^{\infty}\frac{(-1)^{n-1}}{(2n-1)^{15}}=\frac{19194249\pi^{15}}{550071627932302},\sum_{n=1}^{\infty}\frac{(-1)^{n-1}}{(2n-1)^{17}}=\frac{29469\pi^{17}}{8335146508864},
		$$
		$$
		\sum_{n=1}^{\infty}\frac{(-1)^{n-1}}{(2n-1)^{19}}=\frac{\pi^{19}}{2791563952},\sum_{n=1}^{\infty}\frac{(-1)^{n-1}}{(2n-1)^{21}}=\frac{65536\pi^{21}}{1805623744627141}.
		$$
	\end{rmk}
	\begin{rmk}
		By R software, we could compute that
		$$
		\sum_{n=1}^{\infty}\frac{(-1)^{n-1}}{(2n-1)^3}\approx 0.9689461,
		\frac{\pi^3}{32} \approx 0.9689461,
		\sum_{n=1}^{\infty}\frac{(-1)^{n-1}}{(2n-1)^5}\approx 0.9961578,
		\frac{5\pi^5}{1536} \approx 0.9961578,
		$$
		$$
		\sum_{n=1}^{\infty}\frac{(-1)^{n-1}}{(2n-1)^7}\approx0.9995545,\frac{61\pi^7}{184320}\approx0.9995545,\sum_{n=1}^{\infty}\frac{(-1)^{n-1}}{(2n-1)^9}\approx0.9999497,\frac{277\pi^9}{8257536}\approx0.9999497,
		$$
		$$
		\sum_{n=1}^{\infty}\frac{(-1)^{n-1}}{(2n-1)^{11}}\approx0.9999944,\frac{18269\pi^{11}}{5374843438}\approx0.9999944,\sum_{n=1}^{\infty}\frac{(-1)^{n-1}}{(2n-1)^{13}}\approx0.9999994,$$
		
		$$
		\frac{5071747\pi^{13}}{14726725711261}\approx0.9999994,\sum_{n=1}^{\infty}\frac{(-1)^{n-1}}{(2n-1)^{15}}\approx0.9999999,\frac{19194249\pi^{15}}{550071627932302}\approx0.9999999,$$
		$$
		\sum_{n=1}^{\infty}\frac{(-1)^{n-1}}{(2n-1)^{17}}\approx1,\frac{29469\pi^{17}}{8335146508864}\approx1.
		$$
		
	\end{rmk}
	In section 2, we give proofs of  Theorems \ref{thm02}, \ref{thm03}, Remarks \ref{rmk01}, \ref{rmk02}.
	\section{Proof of Theorems \ref{thm02}, \ref{thm03}}
	\begin{proof} [Proof of Theorem \ref{thm02}]By Theorem \ref{thm01}, choose $f(x)=f_k(x)$ with period 2 as follows,
		$$
		\begin{aligned} f_k(x)=\begin{cases} 0& \text{
					if $x\in [-1,0]$,}\\
				x^{2k} & \text{ if $x\in [0,1)$,}
			\end{cases}
		\end{aligned}
		$$
		then
		\begin{align*}
			(f(1+)+f(1-)/2=\frac12=a_0/2+\sum_{n=1}^{\infty}a_n\cos(n\pi),
		\end{align*}
		where
		$$
		a_0=\int_{0}^{1}x^{2k}\dif x=\frac{1}{2k+1},$$
		and we see heuristically that
		\begin{align}\label{03}
			\nonumber a_n&=\int_{0}^{1}x^{2k}\cos(n\pi x)\dif x=\int_{0}^{1}x^{2k}\dif \frac{\sin(n\pi x)}{n\pi}\\
			\nonumber&=\int_{0}^{1}(-2k)\frac{\sin(n\pi x)}{n\pi}x^{2k-1}\dif x=\int_{0}^{1}x^{2k-1}2k\dif\frac{\cos(n\pi x)}{(n\pi)^2}\\
			\nonumber&=\frac{(-1)^n}{(n\pi)^2}2k-\int_{0}^{1}2k(2k-1)\frac{\cos(n\pi x)}{(n\pi)^2}x^{2k-2}\dif x\\
			\nonumber&=\frac{(-1)^n}{(n\pi)^2}2k-\int_{0}^{1}2k(2k-1)x^{2k-2}\dif \frac{\sin(n\pi x)}{(n\pi)^3}\\
			\nonumber&=\frac{(-1)^n}{(n\pi)^2}2k+\int_{0}^{1}2k(2k-1)(2k-2)\frac{\sin(n\pi x)}{(n\pi)^3} x^{2k-3}\dif x\\
			\nonumber&=\frac{(-1)^n}{(n\pi)^2}2k+\int_{0}^{1}2k(2k-1)(2k-2)x^{2k-3}\dif \left(-\frac{\cos(n\pi x)}{(n\pi)^4} \right)\\
			\nonumber&=\cdots=\frac{(-1)^n}{(n\pi)^2}2k-\frac{(-1)^n}{(n\pi)^4}2k(2k-1)(2k-2)+\frac{(-1)^n}{(n\pi)^6}2k(2k-1)(2k-2)(2k-3)\\
			&\quad +\cdots+(-1)^{k-1}\frac{(-1)^n}{(n\pi)^{2k}}(2k)!=\sum_{l=1}^{k}\frac{(-1)^n(2k)!}{(2k+1-2l)!(n\pi)^{2l}}(-1)^{l-1}.
		\end{align}
		In fact, we could obtain (\ref{03}) inductively, 
		here we give a strict proof of (\ref{03}).  First obviously (\ref{03}) holds for $k=1$. Suppose (\ref{03}) holds for $k=m$$(m\ge 1)$. Then for $k=m+1$,
		\begin{align*}
			a_n&=\int_{0}^{1}x^{2(m+1)}\cos(n\pi x)\dif x=\int_{0}^{1}x^{2(m+1)}\dif \frac{\sin(n\pi x)}{n\pi}\\
			&=\int_{0}^{1}(-2(m+1))\frac{\sin(n\pi x)}{n\pi}x^{2m+1}\dif x=\int_{0}^{1}x^{2m+1}(2m+2)\dif\frac{\cos(n\pi x)}{(n\pi)^2}\\
			&=\frac{(-1)^n}{(n\pi)^2}2(m+1)-\int_{0}^{1}(2m+2)(2m+1)\frac{\cos(n\pi x)}{(n\pi)^2}x^{2m}\dif x\\
			&=\frac{(-1)^n}{(n\pi)^2}2(m+1)-\frac{(2m+2)(2m+1)}{(n\pi)^2}\sum_{l=1}^{m}\frac{(-1)^n(2m)!}{(2m+1-2l)!(n\pi)^{2l}}(-1)^{l-1}\\
			&=\sum_{l=1}^{m+1}\frac{(-1)^n(2(m+1))!}{(2(m+1)+1-2l)!(n\pi)^{2l}}(-1)^{l-1},
		\end{align*}
		which implies (\ref{03}) immediately. Hence by Fubini's theorem,
		\begin{align*}
			\frac12&=\frac{1}{4k+2}+\sum_{n=1}^{\infty}\left(\frac{(-1)^n}{(n\pi)^2}2k-\frac{(-1)^n}{(n\pi)^4}2k(2k-1)(2k-2)\right.\\
			&\quad\left.+\frac{(-1)^n}{(n\pi)^6}2k(2k-1)(2k-2)(2k-3)+\cdots+(-1)^{k-1}\frac{(-1)^n}{(n\pi)^{2k}}(2k)!\right)\\
			&=\frac{1}{4k+2}+\sum_{n=1}^{\infty}\sum_{l=1}^{k}\frac{(2k)!}{(2k+1-2l)!(n\pi)^{2l}}(-1)^{l-1}\\
			&=\frac{1}{4k+2}+\sum_{l=1}^{k}\sum_{n=1}^{\infty}\frac{(2k)!}{(2k+1-2l)!(n\pi)^{2l}}(-1)^{l-1}.
		\end{align*}
		The proof is complete.
	\end{proof}
	\begin{proof} [Proof of Theorem \ref{thm03}]By Theorem \ref{thm03}, choose $f(x)=g_k(x)$ with period $2\pi$ as follows,
		$$
		\begin{aligned} g_k(x)=\begin{cases} x^{2k}& \text{
					if $x\in (-\pi,0)$,}\\
				0& \text{
					if $x= 0,\pi,-\pi$,}\\
				0& \text{ if $x\in (0,\pi)$,}
			\end{cases}
		\end{aligned}
		$$
		then
		$$
		a_0=\frac{1}{\pi}\int_{-\pi}^{0}x^{2k}\dif x=\frac{\pi^{2k}}{2k+1},k\ge 0,
		$$
		and heuristically we see that for $k\ge 1$,
		\begin{align}\label{04}
			\nonumber a_n&=\frac{1}{\pi}\int_{-\pi}^{0}x^{2k}\cos(nx)\dif x=\frac{1}{\pi}\int_{-\pi}^{0}x^{2k}\dif \frac{\sin(nx)}{n}\\
			\nonumber&=\frac{1}{\pi}\int_{-\pi}^{0}(-2k)\frac{\sin(nx)}{n}x^{2k-1}\dif x=\frac{1}{\pi}\int_{-\pi}^{0}x^{2k-1}2k\dif\frac{\cos(nx)}{n^2}\\
			\nonumber&=\frac{\pi^{2k-2}2k(-1)^n}{n^2}-\frac{1}{\pi}\int_{-\pi}^{0}\frac{\cos(nx)}{n^2}2k(2k-1)x^{2k-2}\dif x\\
			\nonumber&=\frac{\pi^{2k-2}2k(-1)^n}{n^2}-\frac{1}{\pi}\int_{-\pi}^{0}2k(2k-1)x^{2k-2}\dif \frac{\sin(nx)}{n^3}\\
			\nonumber&=\frac{\pi^{2k-2}2k(-1)^n}{n^2}+\frac{1}{\pi}\int_{-\pi}^{0}2k(2k-1)(2k-2)x^{2k-3}\frac{\sin(nx)}{n^3}\dif x\\
			\nonumber&=\frac{\pi^{2k-2}2k(-1)^n}{n^2}+\frac{1}{\pi}\int_{-\pi}^{0}2k(2k-1)(2k-2)x^{2k-3}\dif \frac{-\cos(nx)}{n^4}\\
			\nonumber&=\cdots=\frac{\pi^{2k-2}2k(-1)^n}{n^2}-\pi^{2k-4}2k(2k-1)(2k-2)\frac{(-1)^n}{n^4}+\cdots\\
			&\quad+\pi^02k(2k-1)\cdots2\frac{(-1)^n(-1)^{k-1}}{n^{2k}}=\sum_{l=0}^{k-1}\pi^{2k-2l-2}\frac{(2k)!(-1)^{n+l}}{(2k-2l-1)!n^{2l+2}}.
		\end{align}
		We now prove (\ref{04}) inductively and strictly. Obviously, (\ref{04}) holds for $k=1$. Suppose that (\ref{04}) holds for $k=m\ge 1$. Then for $k=m+1$,
		\begin{align*}
			a_n&=\frac{1}{\pi}\int_{-\pi}^{0}x^{2(m+1)}\cos(nx)\dif x=\frac{1}{\pi}\int_{-\pi}^{0}x^{2(m+1)}\dif \frac{\sin(nx)}{n}\\
			&=\frac{1}{\pi}\int_{-\pi}^{0}(2m+2)x^{2m+1}\frac{-\sin(nx)}{n}\dif x=\frac{1}{\pi}\int_{-\pi}^{0}(2m+2)x^{2m}\dif \frac{\cos(nx)}{n^2}\\
			&=\frac{\pi^{2m}(2m+2)(-1)^n}{n^2}-\frac{1}{\pi}\int_{-\pi}^{0}(2m+2)(2m+1)x^{2m}\frac{\cos(nx)}{n^2}\dif x\\
			&=\frac{\pi^{2m}(2m+2)(-1)^n}{n^2}-\frac{(2m+2)(2m+1)}{n^2}\sum_{l=0}^{m-1}\pi^{2m-2l-2}\frac{(2m)!(-1)^{n+l}}{(2m-2l-1)!n^{2l+2}}\\
			&=\sum_{l=0}^{m+1-1}\pi^{2(m+1)-2l-2}\frac{(2(m+1))!(-1)^{n+l}}{(2(m+1)-2l-1)!n^{2l+2}},
		\end{align*}
		which results in (\ref{04}).
		Similarly, we see heuristically that for $k\ge 0$,
		\begin{align}\label{05}
			\nonumber b_n&=\frac{1}{\pi}\int_{-\pi}^{0}x^{2k}\sin(nx)\dif x=\frac{1}{\pi}\int_{-\pi}^{0}x^{2k}\dif \frac{\cos(nx)}{n}\\
			\nonumber&=\frac{\pi^{2k-1}(-1)^n}{\pi n}+\frac{1}{\pi}\int_{-\pi}^{0}(2k)\frac{\cos(nx)}{n}x^{2k-1}\dif x\\
			\nonumber&=\frac{\pi^{2k-1}(-1)^n}{\pi n}+\frac{1}{\pi}\int_{-\pi}^{0}2kx^{2k-1}\dif \frac{\sin(nx)}{n^2}\\
			\nonumber&=\frac{\pi^{2k-1}(-1)^n}{\pi n}-\frac{1}{\pi}\int_{-\pi}^{0}2k(2k-1)x^{2k-2} \frac{\sin(nx)}{n^2}\dif x\\
			\nonumber&=\frac{\pi^{2k-1}(-1)^n}{\pi n}+\frac{1}{\pi}\int_{-\pi}^{0}2k(2k-1)x^{2k-2} \dif \frac{\cos(nx)}{n^3}\\
			\nonumber&=\frac{\pi^{2k-1}(-1)^n}{\pi n}-\frac{\pi^{2k-3}2k(2k-1)(-1)^n}{n^3}-\frac{1}{\pi}\int_{-\pi}^{0}2k(2k-1)(2k-2)x^{2k-3} \frac{\cos(nx)}{n^3}\dif x \\
			\nonumber&=\cdots=\frac{\pi^{2k-1}(-1)^n}{\pi n}-\frac{\pi^{2k-3}2k(2k-1)(-1)^n}{n^3}+\cdots+\frac{\pi^{-1}(2k)!(-1)^{n+k}}{n^{2k+1}}-\frac{(-1)^k(2k)!}{\pi n^{2k+1}}\\
			&=\sum_{l=0}^{k}\frac{\pi^{2k-2l-1}(-1)^{n+l}(2k)!}{n^{2l+1}(2k-2l)!}-\frac{(-1)^k(2k)!}{\pi n^{2k+1}}.
		\end{align}
		We now also prove (\ref{05}) inductively. Obviously, (\ref{05}) holds for $k=0$. Suppose (\ref{05}) holds for $k=m\ge 0$. Then for $k=m+1$,
		\begin{align*}
			b_n&=\frac{1}{\pi}\int_{-\pi}^{0}x^{2(m+1)}\sin(nx)\dif x=\frac{1}{\pi}\int_{-\pi}^{0}x^{2(m+1)}\dif \frac{-\cos(nx)}{n}\\
			&=\frac{\pi^{2m+1}(-1)^n}{n}+\frac{1}{\pi}\int_{-\pi}^{0}(2m+2)x^{2m+1}\frac{\cos(nx)}{n}\dif x\\
			\nonumber&=\frac{\pi^{2m+1}(-1)^n}{n}+\frac{1}{\pi}\int_{-\pi}^{0}(2m+2)x^{2m+1}\dif \frac{\sin(nx)}{n^2}
				\end{align*}
				\begin{align*}
			&=\frac{\pi^{2m+1}(-1)^n}{n}-\frac{1}{\pi}\int_{-\pi}^{0}(2m+2)(2m+1)x^{2m}\frac{\sin(nx)}{n^2}\dif x\\
			&=\frac{\pi^{2m+1}(-1)^n}{n}-\frac{(2m+2)(2m+1)}{n^2}\left[\sum_{l=0}^{m}\frac{\pi^{2m-2l-1}(-1)^{n+l}(2m)!}{n^{2l+1}(2m-2l)!}-\frac{(-1)^m(2m)!}{\pi n^{2m+1}}\right]\\
			&=\sum_{l=0}^{m+1}\frac{\pi^{2(m+1)-2l-1}(-1)^{n+l}(2(m+1))!}{n^{2l+1}(2(m+1)-2l)!}-\frac{(-1)^{m+1}(2(m+1))!}{\pi n^{2(m+1)+1}},
		\end{align*}
		which concludes (\ref{05}).
		Therefore, by Theorem \ref{thm01}, Fubini's theorem, (\ref{04}) and (\ref{05}),
		\begin{align*}
			f(\pi/2)=0&=\frac{a_0}{2}+\sum_{n=1}^{\infty}a_n\cos(n\pi/2)+\sum_{n=1}^{\infty}b_n\sin(n\pi/2)\\
			&=\frac{\pi^{2k}}{4k+2}+\sum_{n=1}^{\infty}a_{2n}\cos(n\pi)+\sum_{n=1}^{\infty}b_{2n-1}\sin\left(\left(n-\frac12\right)\pi\right)\\
			&=\frac{\pi^{2k}}{4k+2}+\sum_{l=0}^{k-1}\sum_{n=1}^{\infty}\pi^{2k-2l-2}\frac{(2k)!(-1)^{3n+l}}{(2k-2l-1)!(2n)^{2l+2}}\\
			&\quad +\sum_{l=0}^{k}\sum_{n=1}^{\infty}\frac{\pi^{2k-2l-1}(-1)^{3n-2+l}(2k)!}{(2n-1)^{2l+1}(2k-2l)!}-\sum_{n=1}^{\infty}\frac{(-1)^{k+n-1}(2k)!}{\pi (2n-1)^{2k+1}}\\
			&=\frac{1}{4k+2}\pi^{2k}+\sum_{l=0}^{k-1}\sum_{n=1}^{\infty}\pi^{2k-2l-2}\frac{(2k)!(-1)^n(-1)^l}{(2k-2l-1)!(2n)^{2l+2}}\\
			&\quad+\sum_{l=0}^{k}\sum_{n=1}^{\infty}\frac{\pi^{2k-2l-1}(-1)^{n}(-1)^l(2k)!}{(2n-1)^{2l+1}(2k-2l)!}+\sum_{n=1}^{\infty}\frac{(-1)^k(-1)^n(2k)!}{\pi(2n-1)^{2k+1}},
		\end{align*}
		which implies (\ref{02}) immediately. The proof is complete.
	\end{proof}
	

\end{document}